\documentclass [11pt]{article}
\usepackage[]{amsmath,amssymb,amsfonts,latexsym,amsthm,enumerate,paralist}
\usepackage[numeric,initials,nobysame]{amsrefs}
\date{}
\title{Hamiltonicity thresholds in Achlioptas processes}
\author{{Michael Krivelevich\thanks{School of Mathematical
Sciences, Raymond and Beverly Sackler Faculty
of Exact Sciences, Tel Aviv University, Tel Aviv 69978, Israel. Email address:
{\tt krivelev@post.tau.ac.il}. Research supported in part by
USA-Israel BSF Grant 2006-322, by grant 526/05
 from the Israel Science Foundation, and by a Pazy Memorial Award.}}
 \and {Eyal Lubetzky\thanks{Theory Group of Microsoft Research, One Microsoft Way, Redmond,
WA 98052-6399, USA. Email address: {\tt eyal@microsoft.com}.}}
\and {Benny Sudakov\thanks{Department of Mathematics, UCLA,  Los
Angeles, CA 90095, USA. Email: {\tt bsudakov@math.ucla.edu}.
Research supported in part by NSF CAREER award DMS-0546523, NSF grant DMS-0355497, and by a USA-Israeli BSF grant.}}
}
\newtheorem{theorem}{Theorem}[section]
\newtheorem{lemma}[theorem]{Lemma}

\newtheorem{corollary}[theorem]{Corollary}

\renewcommand{\epsilon}{\varepsilon}

\newcommand{\deq}{:=}
\newcommand{\whp}{\ensuremath{\text{\bf whp}}}
\newcommand{\whps}{\whp\ }
\newcommand{\greedy}{\text{\sc{Greedy }}}

\newcommand{\E}{\mathbb{E}}
\renewcommand{\P}{\mathbb{P}}
\DeclareMathOperator{\var}{Var}
\DeclareMathOperator{\Cov}{Cov}

\newtheoremstyle{upright}%
        {8pt plus2pt minus4pt}%
        {8pt plus2pt minus4pt}%
        {\upshape}%
        {}%
        {\bfseries}%
        {:}%
        {1em}%
        {}%
\theoremstyle{upright}
\newtheorem{remark}[theorem]{Remark}

\newcommand{\ignore}[1]{}
\begin{document}

\maketitle

\begin{abstract}
In this paper we analyze the appearance of a Hamilton cycle in the following random process.
The process starts with an empty graph on $n$ labeled vertices.
At each round we are presented with $K=K(n)$ edges, chosen uniformly at
random from the missing ones, and are asked to add one of them to the
current graph. The goal is to create a Hamilton cycle as soon as possible.

We show that this problem has three regimes, depending on the value of $K$.
For $K=o(\log n)$, the threshold for Hamiltonicity is
$\frac{1+o(1)}{2K}n\log n$, i.e., typically we can construct a Hamilton
cycle $K$ times faster that in the usual random graph process.
When $K=\omega(\log n)$ we can essentially waste almost no edges, and
create a Hamilton cycle in $n+o(n)$ rounds with high probability. Finally,
in the intermediate regime where $K=\Theta(\log n)$, the threshold has
order $n$ and we obtain upper and lower bounds that differ by a multiplicative factor of $3$.
\end{abstract}

\section{Introduction}\label{sec:intro}

The \emph{random graph process}, introduced by Erd\H{o}s and R\'{e}nyi
in their groundbreaking series of papers on random graphs around 1960,
begins with the edgeless graph on $n$
vertices, and at each point adds a single new edge to the current
graph. Each new edge is chosen uniformly and independently out of
all missing edges. Clearly, one may ``freeze'' the random graph
process at a given time-point $t=t(n) \in
\{0,\ldots,\binom{n}{2}\}$, yielding a random graph distributed as
the the well-known Erd\H{o}s-R\'{e}nyi random graph
$\mathcal{G}(n,m)$, which is in turn similar to the binomial
random graph $\mathcal{G}(n,p)$ with
$p=m/\binom{n}{2}$. These two latter models are defined as
follows: $\mathcal{G}(n,p)$ is a random graph on $n$ (labeled)
vertices, where every edge appears with probability $p$, independently of
the other edges, whereas $\mathcal{G}(n,m)$ is a graph uniformly chosen
out of all graphs with $n$ vertices and $m$ edges.  For more information
on these models and the correspondence between them see, e.g.,
\cite{RandomGraphs}, \cite{JLR}.

An important advantage in investigating the random graph process, rather than the two models $\mathcal{G}(n,p)$ and $\mathcal{G}(n,m)$, is that it allows a higher resolution analysis of the appearance of \emph{monotone graph properties} (a graph property is a set of graphs closed under isomorphism; it is monotone if it is also closed under the addition of edges).

A well-known example of this sort is the following \emph{hitting
time} result of Bollob{\'a}s and Thomason \cite{BT} (see also
\cite{RandomGraphs}*{Chapter 7}): the edge that touches the last isolated vertex in the random graph process is typically also the one that makes the graph connected. There are many additional examples for hitting-time results which link natural monotone properties to the minimal degree of the random graph process (cf., e.g., \cite{AKS}, \cite{BHKL}, \cite{Bollobas}, \cite{BolF}, \cite{ER66}, for such results on the appearance of disjoint Hamilton cycles, perfect matchings and the value of the isoperimetric constant).

The above mentioned examples indicate that the main obstacle for
the appearance of many natural graph properties is ``reaching''
the last low-degree vertices. It is therefore natural to ask how
the thresholds for these properties changes if we modify the
random graph process so that we can somehow bypass this obstacle.
The following model that achieves this was proposed by Achlioptas,
inspired by the celebrated ``power of two choices'' result of
\cite{ABKU}: at each step, we are presented with $K \geq 1$
different edges, chosen uniformly and independently out of all
missing edges, and are required to add one of them to our graph.
In this version of the process (which generalizes the
Erd\H{o}s-R\'{e}nyi graph process), one can attempt to either
accelerate the appearance of monotone graph properties, or delay
them, by applying an appropriate online algorithm. It is important
to stress that the process as described above is {\em online} in
nature: the algorithm is denied an ability to see any future edges
at the current round and is forced to make its choice based only on
its previous decisions. While the {\em
offline} versions of these problems are certainly of interest too,
here we restrict ourselves to the online setting.

Quite a few papers have thereafter attempted to settle the many open problems that rise in the above model. These include determining the minimum number of rounds required to ensure the emergence of a giant component, the longest period one may delay its appearance by, delaying the appearance of certain fixed subgraphs and so on (see, e.g., \cite{BF01}, \cite{BFW}, \cite{BK}, \cite{FGS}, \cite{KSS}, \cite{WS}).

In this work, we analyze the optimal thresholds in the Achlioptas process for the appearance of a Hamilton
cycle, a fundamental and thoroughly studied property in random graphs (see, e.g., \cite{RandomGraphs}*{Chapter 8} for further information). In order to obtain some immediate bounds for this problem, recall the well-known fact that the threshold for Hamiltonicity in the random graph $\mathcal{G}(n,m)$ is at $m=(1+o(1))\frac{1}{2}n\log n$ (to be precise, the threshold is at $m=\frac{1}{2}\left(\log n + \log\log n \pm \omega(1)\right)n$, where the $\omega(1)$-term tends to $\infty$ slower than $\log\log n$). Therefore, when presented with $K$ uniformly chosen edges at each round, the minimum number of rounds required for Hamiltonicity is typically between $\frac{1+o(1)}{2K}n\log n$ and $\frac{1+o(1)}{2}n\log n$ (the upper bound is obtained by always selecting the first edge out of the given $K$, whereas the lower bound is obtained by collecting all $K$ edges witnessed). Clearly, for $K \geq \frac{1}{2}\log n$, the above lower bound can be replaced by the trivial lower bound of $n$, as the Hamilton itself consists of $n$ edges.

Our results show that one can indeed construct a Hamilton cycle much faster than in the standard graph process (the above trivial upper bound), essentially matching the two lower bounds mentioned above. In order to formulate this statement precisely, we consider three regimes for the possible values of $K$, and study
the optimal threshold for Hamiltonicity in each of them.

In the first regime, consisting of the sub-logarithmic values of
$K$, we show that the above mentioned lower bound of
$\frac{1+o(1)}{2K}n \log n$ is tight. That is, for every $K =
o(\log n)$, there is an online algorithm that constructs a
Hamilton cycle roughly using the same amount of time it would take
for one to appear when collecting all $K$ edges of every round.
Putting it differently, for such values of $K$ the threshold for
Hamiltonicity is asymptotically $K$ times lower than that of the standard random graph process:
\begin{theorem}\label{thm-sublog}
 Let $K \geq 2$ satisfy $K = o(\log n)$, and consider the
Achlioptas process where $K$ uniformly chosen
new edges are presented at each round.
Then the minimum asymptotical number of rounds
needed for Hamiltonicity is \whps $\frac{1+o(1)}{2K}n \log n$.
\end{theorem}

In the second regime, consisting of the super-logarithmic values
of $K$ (that is, $K=\omega(\log n)$), we show that $n+o(n)$ rounds
suffice for constructing a Hamilton cycle. In other words, for
such values of $K$ it is possible to achieve Hamiltonicity
essentially without waste (the selected edges of almost all rounds
participate in the Hamilton cycle):
\begin{theorem}\label{thm-suplog}
 Let $K$ satisfy $K = \omega( \log n)$, and consider the
Achlioptas process where $K$ uniformly chosen
new edges are presented at each round.
The minimum asymptotical number of rounds
needed for Hamiltonicity is \whps $n+o(n)$.
\end{theorem}

In the intermediate regime, where $K$ has order $\log n$, the
methods we used in order to prove Theorems \ref{thm-sublog} and
\ref{thm-suplog} show that the threshold for Hamiltonicity has
order $n$, and we obtain lower bounds and upper bounds that differ by a multiplicative factor of $3$. This is incorporated in the following theorem.

\begin{theorem}\label{thm-log}
Consider the Achlioptas process where $K$ uniformly chosen
new edges are presented at each round, and $K\to\infty$ with $n$.
Then $\tau_H$, the minimum number of rounds
needed for Hamiltonicity, is at least $(1+o(1))(1 + \frac{\log n}{2K})n$
and at most $(1+o(1))(3+\frac{\log n}{K})n$ \whp. In particular, if $K= \gamma \log n$
for some fixed $\gamma > 0$, then $1 + \frac{1}{2\gamma}+ o(1) \leq \frac{\tau_H}{n} \leq 3 + \frac{1}{\gamma} + o(1)$ \whp.
\end{theorem}

Note that Theorems \ref{thm-sublog},\ref{thm-suplog} and \ref{thm-log} imply a certain discontinuity in the behavior of the ratio between the threshold for Hamiltonicity in the standard random graph process and the corresponding threshold in the Achlioptas process. Indeed, Theorem \ref{thm-sublog} asserts that whenever $K = o(\log n)$, this ratio is asymptotically $K$, whereas Theorem \ref{thm-suplog} implies that for $K=\omega(\log n)$ this ratio is roughly $\frac{1}{2}\log n$. However, for $K=\frac{1}{2}\log n$ for instance, where one might expect this ratio to be roughly $\frac{1}{2}\log n$, Theorem \ref{thm-log} shows that it is in fact roughly between $\frac{1}{10}\log n$ and $\frac{1}{4}\log n$.

The rest of this paper is organized as follows. Sections \ref{sec:sublog},
\ref{sec:superlog} and \ref{sec:intermediate} contain the proofs
of Theorem \ref{thm-sublog} (sub-logarithmic regime), Theorem
\ref{thm-suplog} (super-logarithmic regime) and Theorem
\ref{thm-log} (intermediate regime), respectively. The final section,
Section \ref{sec:conclusion}, is devoted to concluding remarks and
open problems.

Throughout the paper, all logarithms are in the natural basis, and
an event, defined for every $n$, is said to occur {\em with high
probability} (\whp) or {\em almost surely} if its probability tends to $1$
as $n\to\infty$. For a given graph $G$ and a subset of its vertices $S \subset V(G)$, let $S^c \deq V(G) \setminus S$ denote the complement set. Further note that for the sake of simplicity,
our arguments will occasionally move between the random graph models
$\mathcal{G}(n,p)$ and $\mathcal{G}(n,m)$ (for more information on the connection between these models see, e.g., \cite{RandomGraphs}).

\section[Sub-logarithmic regime]{Sub-logarithmic regime: $K=o(\log n)$}\label{sec:sublog}
In this section we prove Theorem \ref{thm-sublog}, that states that for the Achlioptas process with $K=o(\log n)$ edges in each round, the minimal number of rounds required for Hamiltonicity is $\frac{1+o(1)}{2K}n\log n$ \whp.

The lower bound on the asymptotical optimal number of rounds required to obtain a Hamilton cycle
is immediate from the well-known fact that the threshold for Hamiltonicity in $\mathcal{G}(n,m)$ is
$m=\frac{1+o(1)}{2}n\log n$ (see, e.g., \cite{RandomGraphs}*{Chapter 2}). Thus, even if one were allowed to collect
all $K$ edges presented at every round, the minimum asymptotical number of rounds for
Hamiltonicity would still be $\frac{1+o(1)}{2K}n\log n$. It remains to show that this number of rounds is
sufficient to create a Hamilton cycle.

To this end, fix $0 < \epsilon < \frac{1}{100}$ and apply the following algorithm:
\begin{enumerate}[(1)]
  \item \label{alg-phase-1}Construct an expander $H=(V_H,E_H)$ on at least $(1-\epsilon)n$ vertices, in which every nonempty set $S\subset V_H$ of size at most $\frac{n}{100}$ has at least $3|S|$ neighbors in $V_H\setminus S$.

   Cost: $n/\epsilon$ rounds at the most.

  \item \label{alg-phase-2} Build a bipartite expander with parts $V_H,V_H^c$ in which every set $S \subset V_H^c$ of size at most $\frac{n}{100}$ has at least $8|S|$ neighbors in $V_H$.

  Cost: $(1+3\epsilon)\frac{n}{2K} \log n$ rounds at the most.
  \item \label{alg-phase-3} Repeatedly apply the P\'osa rotation-extension technique in order to construct a Hamilton cycle.

   Cost: $2(1+\frac{10^4}{K})n$ rounds at the most.
\end{enumerate}
In what follows, we elaborate on each of the three phases of the algorithm, and show that at the end of Phase \ref{alg-phase-3} the graph contains a Hamilton cycle \whp.

\subsection{Constructing an expander $H$}\label{sec:core}
A \emph{$k$-core} of a graph is its maximum induced subgraph with
minimum degree at least $k$. It is well-known and easy to show that this subgraph is unique, and can be obtained be repeatedly deleting any vertex of degree smaller than $k$ (in any arbitrarily chosen order). In \cite{PSW}, the authors analyze the thresholds for the appearance of a $k$-core in the random graph, as well as its typical size, for any fixed $k \geq 3$.
A simple bound on these quantities will suffice for our purposes, as given in the next lemma:
\begin{lemma}\label{lem-d-core}
Let $D \geq 100$ be an integer, set $p = \frac{3D}{2n}$ and consider the random graph $G \sim \mathcal{G}(n,p)$.
Then the $D$-core of $G$ contains at least $\left(1-\frac{1}{D}\right)n$ vertices \whp.
\end{lemma}
\begin{proof}
We claim that it is enough to show that
 \begin{equation}\label{eq-sparse-cuts-bound}
 |\partial S| \geq n~\mbox{ for any subset $S \subset V(G)$ of size $|S| = n/D$}~,
 \end{equation}
where $\partial S$ denotes the set of all edges that have precisely one endpoint in $S$.
 Indeed, assuming that the $D$-core of $G$ contains less than
 $\left(1-\frac{1}{D}\right)n$ vertices, we may pause the process
 of uncovering the $D$-core after precisely $n/D$ vertices have been deleted. At this point, let $S$ denote the set of deleted vertices. Clearly, each vertex of $S$ has at most $D-1$ neighbors in $V(G)\setminus S$, hence
$|\partial S| \leq (D-1)n/D$ in contradiction to \eqref{eq-sparse-cuts-bound}.

To prove \eqref{eq-sparse-cuts-bound}, fix a set $S$ of
cardinality $n/D$, and recall that $|\partial S|$ is binomially
distributed with parameters $M$ and $p$, where $M =
\frac{D-1}{D^2}n^2$. Therefore, as $\E\left[|\partial S|\right] =
\frac{3(D-1)}{2D}n > n $, the
monotonicity of the binomial distribution between $0$ and its expectation implies that
\begin{align*}
\P(|\partial S|\leq n) &\leq (n+1)\P(|\partial S| = n) = (n+1) \binom{M}{n}p^n (1-p)^{M-n}\leq
(n+1)\left(\frac{\mathrm{e}M}{n}\right)^n p^n \mathrm{e}^{-(1-o(1))pM}\\&\leq
\left(\frac{3(D-1)\mathrm{e}}{2D}\right)^n \mathrm{e}^{-(1-o(1))\frac{3(D-1)}{2D}n} \leq
\mathrm{e}^{-n/14}~,
\end{align*}
where the last inequality holds for any sufficiently large $n$ by our assumption that $D \geq 100$ (with room to spare). Since this assumption on $D$ also implies that the number of possible choices for the set $S$ is
$\binom{n}{n/D} \leq \left(\mathrm{e}D\right)^{n/D} \leq \exp(n/15)$, we conclude that \eqref{eq-sparse-cuts-bound} holds with high probability.
\end{proof}
Recalling that $\epsilon < \frac{1}{100}$, set $D = \lceil
1/\epsilon \rceil$. By the above lemma, performing random
selections in the Achlioptas process for $\frac{3}{4}Dn <
n/\epsilon$ rounds already produces a $D$-core of size at least
$(1-\epsilon)n$ vertices with high probability. Condition on this
event, and throughout the proof let $H$ denote this $D$-core. The
vertex expansion of sets in the induced subgraph on $H$ follows
from basic properties of the random graph, stated in the following simple
lemma.
\begin{lemma}\label{lem-small-avg-deg} Let $k \geq 100$ be a constant, set $p=\frac{3k}{2n}$ and consider the random graph $G \sim \mathcal{G}(n,p)$. Then \whps every induced subgraph on at most $n/25$ vertices of $G$ has average degree at most $k/4$.
\end{lemma}
\begin{proof}
For a subset $S$ of the vertices of size $|S|=s$, let $A_S$ denote the event that the induced subgraph of $G$ on $S$ contains at least
$s k / 8$ edges. Then for any $1 \leq s \leq n/20$ we have
\begin{align*} \P\big(\bigcup_{S:|S|=s} A_S\big) &\leq
\binom{n}{s}\binom{\binom{s}{2}}{s k/8}
p^{s k/8}  \leq \left(
 \frac{\mathrm{e} n}{s}
 \left(\frac{\mathrm{6e}s}{n}\right)^{\frac{k}{8}}
\right)^{s} = \left(
\left(6\mathrm{e}^2\right)^8 \left(6\mathrm{e}\frac{s}{n}\right)^{k-8}
 \right)^{\frac{s}{8}} \\
 &\leq \left(6^{\frac{11}{10}} \mathrm{e}^{\frac{6}{5}}\frac{s}{n}
 \right)^{10s} < \left(\frac{24s}{n}\right)^{10s} ~, \end{align*}
 where the first inequality in the second line follows from the fact that $k-8 \geq 80$ (with room to spare). This implies the  following upper bound on the probability of $\{A_S : |S|\leq n/25\}$:
\begin{align*}
\sum_{s=1}^{n/25}
\left(\frac{24s}{n}
 \right)^{10s} \leq
 \sum_{s=1}^{\log n}
\left(\frac{24\log n}{n}
 \right)^{10s} + \sum_{s=\log n}^{n/25}
\left(\frac{24}{25}
 \right)^{10s}
 <  2\left(\frac{24\log n}{n}\right)^{10} + 25\left(\frac{24}{25}\right)^{10\log n}=o(1)\;,
\end{align*}
as required.
\end{proof}
The above lemma has the following immediate corollary:
\begin{corollary}\label{cor-set-expansion-in-H}
  For $D \geq 100$, set $p=\frac{3}{2}D/n$, let $G\sim \mathcal{G}(n,p)$ and let $H=(V_H,E_H)$ be the $D$-core of $G$. Then $\whp$, every set $S \subset V_H$ of size $1 \leq s \leq n/100$ has at least $3s$
  neighbors in $V_H\setminus S$.
\end{corollary}
\begin{proof}
Condition on the statement of Lemma \ref{lem-small-avg-deg} for $k=D$, and
suppose that some $S \subset V_H$ of size $|S| \leq n/100$
satisfies $|N(S) \cap (V_H \setminus S)| < 3|S|$. In this case,
 the induced subgraph of $G$
on $S \cup N(S)$ contains strictly less than $4|S| \leq n/25$ vertices and at least $D|S|/2$ edges, hence
its average degree is strictly more than $D/4$, contradicting Lemma \ref{lem-small-avg-deg}.
\end{proof}

\begin{remark}\label{rem-connected}
It is easy to verify that the $D$-core constructed above is connected \whp. To see this, recall that our graph is an induced subgraph of a random graph $\mathcal{G}(n,p)$ with $p=\frac{3D}{2n}$ and $D \geq 100$, and furthermore, every nonempty set of size $s \leq n/100$ in the $D$-core has at least $3s$ external neighbors. Therefore, connectivity will immediately follow once we show that $\mathcal{G}(n,p)$, for the above value of $p$, almost surely does not contain any connected component of size $\frac{n}{100} \leq s \leq \frac{n}{2}$. Indeed, this latter fact follows from the following simple calculation:
 \begin{align*}\sum_{s=\frac{n}{100}}^{n/2}\binom{n}{s}(1-p)^{s(n-s)} &\leq \sum_{s=\frac{n}{100}}^{n/2}\left(\frac{\mathrm{e}n}{s}\mathrm{e}^{-p(n-s)}\right)^s
\leq \sum_{s=\frac{n}{100}}^{n/2}\left(100\mathrm{e}^{1-pn/2}\right)^s \leq \sum_{s=\frac{n}{100}}^{n/2}\left(100\mathrm{e}^{-74}\right)^s =o(1)\;.\end{align*}
\end{remark}

\subsection{Constructing a bipartite expander on $(H,H^c)$}\label{sec:bipartite}
In this phase of the algorithm, we create a random bipartite graph
with parts $V_H,V_H^c$, in which the degree of every vertex in
$V_H^c$ is at least $d=20$. Moreover, the neighbors of each vertex
in $V_H^c$ are uniformly distributed over $V_H$. This is achieved
by a \greedy algorithm, comprising two stages, as we next describe.
\begin{enumerate}

\item In the first stage, for $j\in\{0,\ldots,d-1\}$ we attempt to add an edge between $V_H$ and a
vertex of degree $j$ in $V_H^c$, whenever such an edge is presented, settling ambiguities randomly.
Performing this stage for each of the above values of $j$ in a sequence, each time for
$\frac{\epsilon}{2dK}n\log n$ rounds, already suffices in order to provide $1-\frac{\epsilon}{2K}$
fraction of
the vertices of $V_H^c$ with at least $d$ neighbors in $V_H$. This is established in Lemma
\ref{lem-greedy-stage-1} below.

\item Let $X\subset V_H^c$ denote the set of vertices with less than $d$ neighbors in $V_H$ once the
first stage is done. In the second stage, we ``freeze'' the set $X$, and let the \greedy algorithm
prefer edges between $X$ and $V_H$ (disregarding the actual degrees of the vertices in $X$). With high
probability, $(\frac{1}{2}+\epsilon)\frac{n}{K}\log n$ rounds of this stage suffice to provide every
vertex in $X$ with at least $d$ neighbors in $V_H$. This is shown in Lemma \ref{lem-greedy-stage-2}.

\end{enumerate}
\bigskip
\begin{lemma}\label{lem-greedy-stage-1}
Let $d$ be some fixed integer, and $U,W$ be a partition of the
vertices with $|U| \leq |W|$. Consider a \greedy algorithm which,
for each $j\in\{0,\ldots,d-1\}$, performs $T_1=
\frac{\epsilon}{2dK}n\log n$ rounds, where it chooses an edge
between a vertex of degree $j$ in $U$ and a vertex in $W$ whenever
possible, settling ambiguities randomly. Then \whp, the resulting
graph contains at most $\frac{\epsilon}{2K} n$ vertices with
degree smaller than $d$. Moreover, the neighbors of each vertex of
$U$ are uniformly distributed on $W$.
\end{lemma}
\begin{proof}
Since we are establishing an asymptotical bound on the threshold
of the minimal degree, we may allow our algorithm to ignore a
negligible number of rounds. We may thus consider a relaxed
version of the input of each round and allow repeated edges and self loops. That is, at each round we are presented with $K$ ordered pairs chosen independently and uniformly from $[n]^2$. Whenever this selection of $K$ ordered pairs contains either a repeating edge (one that already appears in our graph), or a loop, we ignore this round.
Notice that the probability for this event at a given time $t$ is
at most $K(2t+n)/n^2$. Our analysis focuses on the period $t = O(n
\log n)$ of the graph process, hence for any $K = o(n/\log n)$
this probability is clearly negligible, and the number of rounds we are
ignoring has no affect on our asymptotical upper bound.

Furthermore, recall that by definition, our algorithm selects a random edge out of all those which belong to the cut between $U$ and $W$ (whenever such edges are presented). Clearly, conditioning on the appearance of such edges in a round, each such edge may be treated as the result of two \emph{independent} choices corresponding to the two endpoints: the first is uniformly chosen from $U$, and the second is uniformly chosen from $W$. Thus, since our algorithm decides between these edges solely on the basis of their endpoints in $U$, each edge selected in this manner has an endpoint which remains uniformly distributed in $W$.

It remains to analyze the degrees of the vertices of $U$. Let $X_j=X_j(t)$ denote the set of vertices of $U$ degree $j$ at time $t$, and let $A_t$ denote the event
that an edge between $X_j$ and $W$ appears among the $K$ edges of a round $t$. Then as long as $|X_j| > \frac{\epsilon}{2dK}n$ (and recalling that $|W|\geq n/2$), the probability of $A_t$ satisfies
\begin{equation}
  \label{eq-At-equation-first-st}
  \P(A_t) = 1 - \left(1-\frac{2|X_j||W|}{n^2}\right)^K \geq 1 - \exp\Big(-2K\frac{|X_j||W|}{n^2}\Big) \geq 1-\exp\Big(-\frac{\epsilon}{2d}\Big)~.
\end{equation}
In particular, $\P(A_t)$ is bounded from below by some positive constant. Therefore, for any given
time $t_0$, Chernoff-type concentration results (see, e.g., \cite{ProbMethod}*{Appendix A}) imply the
following. For some constant $c_1 > 0$, performing $\Delta_0 \deq c_1 |X_j(t_0)|$ rounds either reduces
$|X_j(t_0+\Delta_0)|$ below $\frac{\epsilon}{2dK}n$, or with probability at least
$1-\exp(-\Omega(n/K))$
yields at least $2|X_j(t_0)|$ random edges incident to $X_j(t_0)$. Henceforth, condition on this event.

The classical balls and bins experiment asserts that, when throwing $r \cdot m$ balls independently and
uniformly onto $m$ bins, where $r > 0$ is fixed and $m\to\infty$, the distribution of the fraction of
the bins with precisely $\ell$ balls ($\ell=0,1,\ldots$) converges to a Poisson distribution with mean
$r$ (see, e.g., \cite{Feller}, \cite{JK}). In particular, the expectation and variance of the number of empty bins
in the above experiment tend to $m \mathrm{e}^{-r}$ and $O(m)$ respectively. Applying this to our
setting, where $r = 2$ and $m = |X_j(t_0)|\geq \frac{\epsilon}{2dK}n$, we deduce that the size of
$X_j(t_0+\Delta_0)$ is reduced to at most $|X_j(t_0)|/\mathrm{e}$ with probability at least $1-O(1/m)
\geq 1 - O(K/n)$.

Repeating this argument for $t_l = t_{l-1}+\Delta_{l-1}$, $\Delta_l=c_1 |X_j(t_l)|$ and $l=1,\ldots,\log
n$ (accumulating the individual error probabilities of $O(K/n)$ easily allows this number of
repetitions) we deduce that after $$\sum_l \Delta_l \leq c_1 |X_j(t_0)| \sum_l \mathrm{e}^{-l} \leq 2c_1
n$$ rounds, $|X_j| < \frac{\epsilon}{2dK}n$ \whp. Since $K=o(\log n)$, we indeed have
$\epsilon\frac{n}{2dK} \log n = \omega(n)$ rounds at our disposal for this stage.

By applying the same argument for $j$ in $\{1,\ldots,d-1\}$, we obtain that at the end of
$\frac{\epsilon}{2K}n\log n$ rounds, \whp $~|\cup_{j<d}X_j| \leq \frac{\epsilon}{2K} n$,
completing the proof. \end{proof}
\medskip
\begin{lemma}\label{lem-greedy-stage-2}
Let $d$ be some fixed integer and $U,W$ be two disjoint sets of vertices with $|U| \leq \frac{\epsilon}{2K} n$ and $|W| \geq (1-\epsilon)n$. Consider a \greedy algorithm which chooses an edge between vertices $U$ and $W$ whenever one appears (otherwise, this round is ignored), settling ambiguities randomly. Then performing this algorithm for $T_2=(\frac{1}{2}+\epsilon)\frac{n}{K}\log n$ rounds gives a graph where, \whp, every vertex of $U$ has at least $d$ neighbors uniformly distributed over $W$.
\end{lemma}
\begin{proof}
As in the proof of Lemma \ref{lem-greedy-stage-1}, we may assume that each rounds presents $K$ ordered pairs, uniformly and independently chosen from $[n]^2$. Recall that whenever the algorithm selects an edge between $U$ and $W$, its two endpoints are uniformly distributed over $U$ and $W$ respectively.

Letting $A_t$ denote the probability of witnessing an edge between $U$ and $W$ at time $t$, we have:
\begin{align}
\P(A_t) &= 1-\left(1-\frac{2|U||W|}{n^2}\right)^K \geq 1-\Big(1-2(1-\epsilon)\frac{|U|}{n}\Big)^{K} \nonumber\\ &\geq 2(1-\epsilon)K\frac{|U|}{n} \left(1 - K\frac{|U|}{n}\right) \geq
(2-3\epsilon)K\frac{|U|}{n}~,  \label{eq-At-equation-second-st}
\end{align}
where the first inequality in the second line is by the well-known fact that
$(1-p)^k \leq 1-kp+\binom{k}{2}p^2$ for any $0<p < 1$ and integer $k \geq 2$, and the last inequality is
by our assumption on the size of $U$.

 Therefore, the total number of rounds with edges incident to $U$ along $T_2=
 (\frac{1}{2}+\epsilon)\frac{n}{K}\log n$ consecutive rounds
 stochastically dominates a binomial
 random variable with mean $$T_2 \cdot (2-3\epsilon)K |U|/n =
 \left(1+\mbox{$\frac{1}{2}$}\epsilon-3\epsilon^2\right)|U|\log n \geq
 (1+\mbox{$\frac{2}{5}$}\epsilon)|U|\log n~,$$
 where in the last inequality we used the fact that $\epsilon \leq
 1/30$. It follows that
the probability of observing at least $M =
(1+\frac{\epsilon}{3})|U|\log n$ edges incident to $U$ along those $T_2$
rounds is at least
$1-n^{-\Omega(|U|)}$. Condition on this event.

 Let $X(t) \subset U$ denote the vertices of $U$ which have degree smaller than $d$ at time $t$.
 Since any edge selected by the algorithm has one endpoint uniformly and independently distributed on $U$, we deduce that, for large $n$, any $v \in U$ satisfies
\begin{align*}\P\left(v \in X(T_2)\right) &\leq \P\left( \mathrm{Bin}(M, 1/|U|) < d\right)
  \leq d \cdot\P\left(\mathrm{Bin}(M,1/|U|) = d-1\right) \\
  &\leq d\binom{M}{d-1}|U|^{-(d-1)}
  \mathrm{e}^{-(1-o(1))\frac{M}{|U|}}
  \leq d\left(\frac{(1+\frac{\epsilon}{3})\mathrm{e}}{d-1}\log n\right)^{d-1} n^{-1-\frac{\epsilon}{3}+o(1)}\leq n^{-1-\epsilon/4}~,
\end{align*}
where the last inequality in the second line holds for any sufficiently large $n$, and the last inequality in the first line is by the monotonicity of the binomial distribution, as
$\frac{M}{|U|} = (1+\frac{\epsilon}{3})\log n > d$ for any sufficiently large $n$.

It follows that $\E\left[ |X(T_2)|\right] \leq n^{-\epsilon / 4}$, and thus $X(T_2)$ is \whps empty, as required.
\end{proof}

Once we apply Lemmas \ref{lem-greedy-stage-1} and \ref{lem-greedy-stage-2}, we will have obtained a
random bipartite graph on $V_H,V_H^c$, whose expansion will follow from the next lemma:

\begin{lemma}\label{lem-set-expansion-in-H-bar}
Let $G$ be the following random bipartite graph on
$U,W$: $|U|\geq \frac{2}{3}n$, $|W| \leq n$ for some integer $n$, and every vertex of $W$ has $20$
neighbors independently and uniformly chosen from $U$. Then \whp, every set $S \subset W$ of size $1
\leq s \leq n/100$ satisfies $|N(S) \cap U| > 8s$. \end{lemma} \begin{proof} Let $u=|U|$, and let $1
\leq s \leq n/100$. The probability that there exists a subset $S\subset W$ of cardinality $s$ such that
$|N(S) \cap U| \leq 8s$, is at most \begin{align*}&\binom{n}{s}\binom{u}{8s}
\left(\binom{8s}{20}/\binom{u}{20}\right)^s
\leq \left(\left(\frac{\mathrm{e}n}{s}\right)\left(\frac{\mathrm{e}u}{8s}\right)^8
\left(\frac{8s}{u-19}\right)^{20}\right)^{s} = \left((1+o(1))\frac{8^{12}
\mathrm{e}^{9}s^{11}}{u^{12}/n} \right)^{s}
 \\
 &\leq
 \left((1+o(1))\left(2^{24} \cdot 3^{12}\cdot\mathrm{e}^9\right)^{\frac{1}{11}} \cdot\frac{s}{n}
\right)^{11s} \leq \left(\frac{50s}{n}\right)^{11s}~,
 \end{align*}
where the first inequality in the last line is by the assumption that $u \geq 2n/3$, and the inequality following it
holds for any sufficiently large $n$. Therefore, the following calculation, similar to the one made in Lemma \ref{lem-small-avg-deg}, gives that the probability that there exists a set $S$ of size $1 \leq s \leq n/100$ that satisfies $|N(S)\cap U| \leq 8 s$ is at most
\begin{align*}
\sum_{s=1}^{n/100}
\left(\frac{50s}{n}
 \right)^{11s} \leq
 \sum_{s=1}^{\log n}
\left(\frac{50\log n}{n}
 \right)^{11s} + \sum_{s=\log n}^{n/100}
2^{-11s}
 <  2\cdot \left(\frac{50\log n}{n}\right)^{11} + 2\cdot 2^{-11\log n}=o(1)~.
\end{align*}
This completes the proof of the lemma.
\end{proof}

Apply Lemma \ref{lem-greedy-stage-1} with $d=20$, $W=V_H$ and $U=V_H^c$, followed by Lemma
\ref{lem-greedy-stage-2} with $W=V_H$ and $U$ being the remaining vertices of $V_H^c$ of degree smaller
than $20$.  Then Lemma \ref{lem-set-expansion-in-H-bar} (with $V_H,V_H^c$ playing the roles of $U,W$
respectively) yields: \begin{corollary}\label{cor-set-expansion}
  The resulting graph $G$ satisfies the following \whp: every set $S \subset V(G)$ of size at most
$n/100$ has strictly more than $2|S|$ neighbors in $V(G) \setminus S$.
\end{corollary} \begin{proof}
  Let $S \subset V(G)$ be a set containing at most $n/100$ vertices. If $|S \cap V_H| > \frac{2}{3}|S|$,
  Corollary \ref{cor-set-expansion-in-H} implies that $S$ has strictly more than
$3|S \cap V_H |\geq 2|S|$ neighbors in
  $V_H \setminus S$ alone. Otherwise, $|S \cap V_H^c| \geq |S|/3$, hence Lemma
  \ref{lem-set-expansion-in-H-bar} gives \begin{align*}
    |N(S \cap V_H^c) \cap (V_H \setminus S)| &> 8 |S \cap V_H^c| - (|S| - |S \cap V_H^c|) = 9|S \cap
    V_H^c| - |S| \geq 2|S|~.\qedhere
  \end{align*}
\end{proof}

\subsection{P\'osa's rotation-extension technique}\label{rot-ext}

As we already mentioned in the introduction, a key tool of our proof is
the celebrated rotation-extension technique, developed by P\'osa \cite{Posa} and applied in several subsequent papers on Hamiltonicity of random and pseudo-random graphs (cf., e.g., \cite{BFF},\cite{FK},\cite{KS},\cite{SV}). Below we will cover this approach, including a key lemma and its proof.

Let $P=x_0x_1\ldots x_h$ be a longest path in a graph $G=(V,E)$,
starting at a vertex $x_0$. Suppose $G$ contains an edge $(x_i,x_h)$ for
some $0\le i<h$. Then a new path $P'$ can be obtained by rotating the
path $P$ at $x_i$, i.e. by adding the edge $(x_i,x_h)$ and erasing
$(x_i,x_{i+1})$. This operation is called an {\em elementary rotation}.
Note that the obtained path $P'$ has the same length
$h$ (here and in what follows we measure path lengths in edges and not
in vertices) and starts at $x_0$. We can therefore apply an elementary
rotation to the newly obtained path $P'$, resulting in a path
$P''$ of length $h$, and so on. If after a number of rotations an
endpoint $x$ of the obtained path $Q$ is connected by an edge to a
vertex $y$ outside $Q$, then $Q$ can be extended by adding the edge
$(x,y)$.

The power of the rotation-extension technique of P\'osa hinges on the
following lemma.

\begin{lemma}\label{lem-Posa}
Let $G$ be a graph, $P$ a longest path in $G$ and ${\cal P}$ the set of
all paths obtainable from $P$ by a sequence of elementary rotations.
Denote by $R$ the set of ends of paths in ${\cal P}$, and by $R^-$ and
$R^+$ the sets of vertices immediately preceding and following the
vertices of $R$ on $P$, respectively. Then
$
(N(R)\setminus R) \subset R^-\cup R^+
$.
\end{lemma}

\begin{proof} Let $x\in R$ and $y\in V(G)\setminus (R\cup R^-\cup R^+)$,
and consider
a path $Q\in{\cal P}$ ending at $x$. If $y\in V(G)\setminus V(P)$, then
$(x,y)\not\in E(G)$, as otherwise the path $Q$ can be extended by adding
$y$, thus contradicting our assumption that $P$ is a longest path.
Suppose now that $y\in V(P)\setminus (R\cup R^-\cup R^+)$. Then $y$ has
the same neighbors in every path in ${\cal P}$, because an elementary
rotation that removed one of its neighbors along $P$ would, at the same
time, put either this neighbor or $y$ itself in $R$ (in the former case
$y\in R^-\cup R^+$).  Then if $x$ and $y$ are adjacent, an elementary
rotation applied to $Q$, produces a path in ${\cal P}$ whose endpoint is
a neighbor of $y$ along $P$, a contradiction. Therefore in both cases $x$
and $y$ are non-adjacent.
\end{proof}

We will use the following immediate consequence of Lemma \ref{lem-Posa}, where the length of
a simple cycle or a simple path is defined to be the number of edges it contains.
\begin{corollary}\label{cor-Posa2}
Let $h,r$ be positive integers. Let $G=(V,E)$ be a graph such that its
longest path has length $h$, but it contains no cycle of length $h+1$.
Suppose furthermore that for every set $R\subset V$ with $|R|<r$ we
have $|N(R)|\ge 2|R|$. Then there are at least $r^2/2$ non-edges in
$G$ such that if any of them is turned into an edge, then the new graph
contains an $(h+1)$-cycle.
\end{corollary}
\begin{proof} Let $P=x_0x_1\ldots x_h$ be a longest path in $G$ and let
$R,R^-,R^+$ be as in Lemma \ref{lem-Posa}. Notice that $|R^+|\le |R|-1$ and
$|R^-|\leq |R|$, since $x_h \in R$ has no following vertex on $P$ and thus does not contribute
an element to $R^+$.

According to Lemma \ref{lem-Posa},
 $$|N(R)\setminus R| \leq |R^-\cup
R^+| \leq 2|R| - 1~,$$ and it follows that $|R|\ge r$. Moreover, $(x_0,v)$ is not an edge for any $v\in R$
(there is no $(h+1)$-cycle in the graph), whereas adding any edge $(x_0,v)$ for $v\in R$
creates a $(h+1)$-cycle.

Fix a subset $\{y_1,\ldots,y_r\}\subset R$.
For every $y_i\in R$, there is a path $P_i$ ending at $y_i$, that can be
obtained from $P$ by a sequence of elementary rotations. Now fix $y_i$
as the starting point of $P_i$ and let $Y_i$ be the set of endpoints of
all paths obtained from $P_i$ by a sequence of elementary rotations. As
before, $|Y_i|\ge r$, no edge joins $y_i$ to $Y_i$, and adding any such
edge creates a cycle of length $h+1$. Altogether we have found $r$
pairs $(y_i,Y_i)$. As every non-edge is counted at most twice,
the conclusion of the lemma follows.
\end{proof}

The reason we are after a cycle of length $h+1$ in the above argument is
that if $h+1=n$, then a Hamilton cycle is created. Otherwise, if the graph is
connected, then there will be a path $L$ connecting a newly created
cycle $C$ of length $h+1$ with a vertex outside $C$. Then combining $C$
and $L$ in an obvious way creates a longer path in $G$. Indeed, in our case the graph is \whps connected, as it comprises $H$, the $D$-core (for some $D \geq 100$) of a random graph $\mathcal{G}(n,p)$ for $p=\frac{3D}{2n}$, and an additional set of vertices $V_H^c$, in which every vertex has at least 20 neighbors in $H$. Thus, the connectivity immediately follows the fact that $H$ itself is connected \whp, as argued in Remark \ref{rem-connected}.

We can now return to our setting of the Achlioptas process: Corollary
\ref{cor-set-expansion} implies that once Phase \ref{alg-phase-2} is complete, with high probability the requirements of Corollary \ref{cor-Posa2}
are met with $r=\frac{n}{100}$. Condition now that this is indeed the case. Thus, at any given round from this point on, either the graph is Hamiltonian, or there are
at least $r^2/2$ pairs of vertices such that adding any of them to the
graph will increase its maximum length of all paths by at least $1$.
Further recall that we have $O(n)$ additional rounds at our disposal for
Phase \ref{alg-phase-3}; we
wish to state that this amount of rounds will almost surely suffice for up to $n$ applications of Corollary \ref{cor-Posa2} (each time on a possibly different edge set). An easy formulation of this statement, which does not involve stopping times, is the following: let $G_0$ denote the initial graph (at the beginning of phase \ref{alg-phase-3}), set $T \deq \lceil 10^4/K\rceil$ and perform the following trials for $t \in \{1,\ldots,2n\}$:
\begin{itemize}
  \item If $G_{t-1}$ is already Hamiltonian, set $G_t = G_{t-1}$; the trial is successful.
  \item If $G_{t-1}$ does not contain a Hamilton cycle:
   \begin{itemize}\item Let $\mathcal{E}_t$ denote the above
   mentioned set of pairs, each of which would increase the length
   of a maximum path in $G_{t-1}$ by at least $1$, or would create
   a Hamilton cycle.
   \item Perform $T$ rounds of the Achlioptas process, selecting an edge from $\mathcal{E}_t$ whenever possible (and settling ambiguities arbitrarily). Let $G_t$ denote the resulting graph.
   \item The trial is successful iff $G_t$ contains at least one of the edges of $\mathcal{E}_t$.
   \end{itemize}
\end{itemize}
Clearly, the failure probability of the above trial is at most
\begin{align*} \P(\mbox{missing $\mathcal{E}_t$ in $T$ given rounds}) &\leq \left(1-|\mathcal{E}_t|/\binom{n}{2}\right)^{K T} \leq \exp\left(-K T \cdot \left(\frac{r}{n}\right)^2\right)\leq \mathrm{e}^{-1}~.\end{align*}
Let  $X$ be the number of successful trials in the above defined sequence; then
$X$ stochastically dominates a binomial random variable $X'\sim\operatorname{Bin}(2n,1-\mathrm{e}^{-1})$,
and Chernoff's inequality implies that $\P(X \geq n) \geq 1-\exp(-\Omega(n))$.
Recall that as long as the graph is not Hamiltonian, every
successful trial increases the length of a longest path by at
least $1$ or creates a Hamilton cycle. Therefore, after $n$ successful
trials the graph surely contains a Hamilton cycle.
Altogether, the above sequence of trials utilized $2Tn \leq 2\left(1+\frac{10^4}{K}\right)n$ rounds in order to generate a Hamiltonian cycle $\whp$.

\section[Super-logarithmic regime]{Super-logarithmic regime: $K=\omega(\log n)$}\label{sec:superlog}
Next, we prove Theorem \ref{thm-suplog}, which states that, for the Achlioptas process with $K=\omega(\log n)$ edges in each round, the minimal number of rounds required for Hamiltonicity is $n+o(n)$ \whp.
The lower bound in this setting is obtained by the Hamilton cycle
alone, and it is left to provide an algorithm that attains this
bound asymptotically.

Throughout the proof, write
$$K = h^{10} \log n~,~m\deq \lfloor n/h^2 \rfloor~,$$
where $h=h(n)$ tends to infinity with $n$ by the assumption on $K$. Moreover, it will be convenient to assume that $h=o(\log n)$ (we can always restrict ourselves to a prefix of the sequence of edges given in a round).
Following is an outline of the algorithm we will use:
\begin{enumerate}
  \item \label{item-disjoint-paths}  Cover $n-m$ vertices by at most $n/(h^4\log n)$ disjoint simple paths, via a \greedy algorithm. Let $\mathcal{P}$ denote this set of paths, and let $Y=Y(\mathcal{P})$ be the set of all vertices in $\mathcal{P}$.

      Cost: $\left(1+\mathrm{e}^{-h}\right)n=(1+o(1))n$ rounds.

  \item \label{item-res-expander}Construct an expander on some subset $X \subset Y^c$ of size $|X| \geq \frac{2}{3}m$, with the following property: for any subset $A \subset X$ of size at most $m/200$, there exists
      a subset $X' \subset X \setminus A$ of size at least $|X|-2|A|$, such that the induced subgraph on $X'$ has diameter at most $3\log m$.

      Cost: $1500m = o(n)$ rounds.

  \item \label{item-join-paths} Using the above properties of $X$, we repeatedly connect the endpoints of two paths in $\mathcal{P}$ using a simple path of length at most $3\log m$ in $X\setminus Y$. The two paths are removed from $\mathcal{P}$, and the new path is added to $\mathcal{P}$ in their place (the set $Y$ is updated accordingly).  At the end of the stage, $\mathcal{P}$ contains a single path, which is thereafter closed into a simple cycle on $|Y|$ vertices.

  Cost: $5n/h^7 = o(n)$ rounds.

  \item \label{item-small-ham-cycle} Apply the algorithm for the sub-logarithmic regime restricted to the induced subgraph on $Y^c$, in order to turn it into a Hamiltonian graph.

  Cost: $m h = o(n)$ rounds.

  \item \label{item-join-cycles} Merging the two simple cycles of Phases \ref{item-join-paths},\ref{item-small-ham-cycle} to a single Hamilton cycle.

      Cost: $m + n/\log n = o(n)$ rounds.
\end{enumerate}

We now discuss each phase of the algorithm in more details. As argued in our analysis of the sub-logarithmic regime (Section \ref{sec:sublog}), we may assume that the selection of edges in each round consists of $K$ ordered pairs, independently and uniformly chosen over $[n]^2$. For the sake of convenience, this will indeed be our assumption throughout this section.

\subsection[Phase 1: covering most vertices by disjoint simple paths]{Phase \ref{item-disjoint-paths}: covering most vertices by disjoint simple paths}
This phase is accomplished by the following simple \greedy algorithm. Take an arbitrary subset of $L \deq \lfloor n/(h^4 \log n) \rfloor$ vertices, serving as $L$ trivial disjoint simple paths. At each step, we attempt to add an edge between an endpoint of one of these $L$ paths and the set of the remaining vertices (excluding the vertices on the paths), thus extending one of the paths. It is easy to show that this can be performed repeatedly, as long as there are at least $m$ vertices beyond those which belong to the $L$ given paths.
\begin{lemma}
  \label{lem-greedy-path-cover}
Consider the Achlioptas process on $n$ vertices with $K=h^{10} \log n$ edges per round. Then \whp,
a \greedy algorithm can cover $n-m$ vertices with at most $L=\lfloor n/(h^4\log n) \rfloor$ disjoint simple paths utilizing at most $\left(1+\mathrm{e}^{-h}\right)n$ rounds.
\end{lemma}
\begin{proof}
Let $Y$ denote a set of $L$ arbitrarily chosen vertices, and treat these vertices as the \emph{left} endpoints of (trivial) disjoint simple paths.
At each step, we will attempt to join the left endpoint of one of these paths to a new vertex of $Y^c$, thus
increasing the length of the corresponding path by $1$, while adding a new vertex to $Y$. In this case, the newly added vertex would become the new left endpoint of its path.
As long as $|Y^c| > n/h^2$, the probability
of witnessing an edge accomplishing this task satisfies
\begin{align*}\P(\mbox{extending $Y$}) &=1-\left(1-\frac{2L|Y^c|}{n^2}\right)^K \geq 1-\exp\left(-2K L |Y^c| / n^2\right) \\
&\geq 1-\exp\left(-2L\frac{h^8\log n}{n}\right) \geq 1-\mathrm{e}^{-h^4}~, \end{align*}
where the last inequality holds for a sufficiently large $n$. Stochastically bounding the above process by the corresponding binomial variable, and applying standard concentration arguments, we deduce that $\left(1+\mathrm{e}^{-h}\right)n$ rounds easily suffice to construct a path cover as required \whp.
\end{proof}

\begin{remark}
  This phase is the most time consuming one of the algorithm --
  all other phases take typically $o(n)$ rounds.
\end{remark}

\subsection[Phase 2: constructing an expander on the remaining vertices]{Phase \ref{item-res-expander}: constructing an expander on the remaining vertices}
In this phase, we consider $Y^c$, the $m$ vertices that were not covered by paths in the previous phase. We will construct an expander on a subset $X\subset Y^c$, which will serve as a ``connector'' for the paths in $Y$, in the following sense: We wish to repeatedly join two paths in $Y$ using a path in $X$, then delete this path from $X$ and repeat the process. To this end, the induced subgraph on $X$ should have a small diameter, and furthermore, this property should be retained even after deleting a small fraction of its vertices. This is established by the next lemma.

\begin{lemma}\label{lem-resilient-expander}
Consider the Achlioptas process on $n$ vertices with $K=h^{10}
\log n$ edges per round, and let $Y^c$ denote a fixed set of $m = n / h^2$ vertices. Then \whp, $1500m$ rounds suffice to construct an expander on some subset $X\subset Y^c$ of size at least $0.999m$, with the following property: For any set $A \subset X$ of size at most $m/200$, there exists a set $B \subset X \setminus A$ of
  size at most $|A|$ such that $X \setminus (A \cup B)$ has diameter at most $\log m$.
\end{lemma}
\begin{proof}
At a given round, the probability to witness an edge with both endpoints in $Y^c$ is at least
$$1-\left(1-\frac{m(m-1)}{n^2}\right)^K = 1-\exp(-(1-o(1))K(m/n)^2) = 1 - n^{-\Omega(h^6)}~,$$
hence along any given $n$ rounds, \whps every round will contain at least one such edge. Clearly, given this event, the first such edge witnessed is uniformly distributed in $Y^c$. Following Lemma \ref{lem-d-core}, take $D = 2000$, and perform
$\frac{3}{4}D m$ rounds, where the first edge in a round with both endpoints in $Y^c$ is selected. Since this precisely establishes a random graph $G\sim \mathcal{G}(m,\frac{3D}{2m})$ on the vertices of $Y^c$, Lemma \ref{lem-d-core} ensures us that the $D$-core of this graph has size at least $(1-\frac{1}{D})m$ \whp. Let $X$ denote this $D$-core, and recall that Corollary \ref{cor-set-expansion-in-H} asserts that every set $S\subset X$ of size $1 \leq s \leq m/100$ has at least $3s$ neighbors in $X\setminus S$.

We next claim that, \whp, any two sets of size $s=m/200$ of $X$ have an edge between them. Indeed, it is enough to show this on the original random graph $G \sim \mathcal{G}(m,\frac{3D}{2m})$, since $X$ is an induced subgraph of $G$. The following simple calculation shows this fact:
\begin{align*} \P\bigg(\exists A,B: &\Big\{\begin{array}{l}N(A) \cap B = \emptyset\;,\\|A|=|B|=s\end{array}\bigg) \leq \binom{m}{s}^2
(1-p)^{s^2} \leq \left(\frac{\mathrm{e}m}{s}\mathrm{e}^{-ps/2}\right)^{2s}  \leq \left(200\mathrm{e}^{1-\frac{3D}{800}}\right)^{2s}=o(1)~,
\end{align*}
where the last equality holds as $D=2000$.

Let $A\subset X$ denote an arbitrary (nonempty) subset of the vertices of $X$ of size at most $m/200$. Set $X' = X \setminus A$ and let $B \subset X'$ denote the maximum set of size at most $m/100$ such that the following holds: $|N(B) \cap \left(X \setminus (A\cup B)\right)|< 2|B|$. Then a simple calculation will show that $|B| < |A| \leq m/200$. Indeed, if $B$ is nonempty and $|B|\leq m/100$, it has at least $3|B|$ neighbors in $X\setminus B$, and therefore
$$3|B| \leq |N(B)\cap (X\setminus B)| \leq |N(B) \cap \left(X \setminus (A\cup B)\right)| + |A| < |A| + 2|B|~.$$

Furthermore, consider the set of vertices $X'' = X \setminus (A\cup B)$, and suppose that some nonempty set $C \subset X''$
of size at most $m/200$ satisfies $|N(C) \cap (X''\setminus C)| < 2|C|$.
Then clearly the set $B \cup C$ has less than $m/100$ vertices and less than $2|B\cup C|$ neighbors in $X\setminus(A\cup B \cup C)$, contradicting the maximality of $B$. We deduce that every nonempty set $C$ of $X''$ of size at most $m/200$ has at least $2|C|$ neighbors in $X''\setminus C$.

It remains to show that the diameter of $X''$ is at most $\log m$.
This quickly follows from the properties we have already
established on $X$. Let $u,v$ be two vertices of $X''$, and set
$t=\lfloor\log_2 m\rfloor$. Since every nonempty set of vertices
$C \subset X''$ of size $|C|\leq m/200$ has at least $2|C|$
neighbors in $X''\setminus C$, it follows that the
$t$-neighborhood of $u$ contains at least $m/200$ vertices, and
the same applies to $v$. Finally, either these two neighborhoods
intersect, or we can select an arbitrary subset of $m/200$
vertices from each of them, and obtain an edge between them.
Altogether, the distance between $u$ and $v$ is at most
$2\log_2 m + 1 \leq 3\log m$ (for every sufficiently large $m$).
\end{proof}

\subsection[Phase 3: concatenating the paths into a cycle via the expander]{Phase \ref{item-join-paths}: concatenating the paths into a cycle via the expander}
Recall that we have a collection $\mathcal{P}$ of $ m/(h^4 \log n) = o(m/\log n)$ simple paths covering the vertices of $Y$, and that $X\subset Y^c$ is a subset of size
$m/2\leq |X| \leq m$, such that the induced subgraph on $X$ has the following property: upon removing a negligible  subset of its vertices, it still contains an induced subgraph on $(1-o(1))|X|$ vertices with a diameter of at most $3\log m = O(\log n)$. Thus, once we obtain two edges connecting endpoints of two paths in $\mathcal{P}$ to $X$, we can concatenate these paths using a path of length at most $3\log m$ from $X$, update $X$ and $Y$ accordingly, and continue the process. Crucially, at each point, we will have removed at most $o(m)$ vertices from $X$, thus the above expansion property is maintained.

\begin{lemma}\label{lem-path-joining}
Let $K$ and $X,Y$ be as above, and suppose that $Y$ is covered by $L \leq n/(h^4\log n)$ disjoint simple paths. Then a simple cycle going through every vertex in $Y$ can be constructed in at most $5n/h^7$ rounds \whp.
\end{lemma}
\begin{proof}
Let $\mathcal{P}$ denote the given set of paths covering $Y$, and for every $i \geq 0$ let $a_i \deq L 2^{-i}$.
Given a subset $I\subset X$ of $o(m)$ vertices (initially defined to be the empty set), we will write $\tilde{X} \subset X\setminus I$
as the subset of size $(1-o(1))|X|$, on which the induced subgraph has diameter at most $3\log m$ (as ensured by Lemma \ref{lem-resilient-expander}).
As long as $a_i \geq |\mathcal{P}| >  a_{i+1} \geq 1$, consider the following process:
 \begin{itemize}
   \item Perform $T_i \deq \lceil 2 n h^2 / (a_{i} K)\rceil $ rounds, preferring an edge between $\tilde{X}$ and an endpoint of some path in $\mathcal{P}$ whenever possible (settling ambiguities randomly).
       If indeed such edges were witnessed (and thus selected) in this stage, let $P_1\in\mathcal{P}$ denote the first path whose endpoint was connected to $X$.
   \item Repeat the first step (performing $T_i$ rounds), this time attempting to
   connect an endpoint of some path in $\mathcal{P} \setminus P_1$ (assuming this set of paths is nonempty) to $X$.
\end{itemize}
The probability of not being able to connect any $P_1 \in \mathcal{P}$ to $\tilde{X}$ in the first step above is at most
$$  \left(1-\frac{4|\tilde{X}||\mathcal{P}|}{n^2}\right)^{K T_i} \leq
\left(1-\frac{4\cdot (1-o(1))|X|\cdot a_{i+1}}{n^2}\right)^{K T_i} \leq
\exp\left(-(2-o(1))T_i \frac{a_{i+1} K}{n h^2}\right) \leq 1/\mathrm{e}~,$$
where we used the facts that $|X| \geq m/2$, $a_i = 2 a_{i+1}$, $m=n/h^2$ and there are $2|\mathcal{P}|$ endpoints of paths. Since $|\mathcal{P}|-1 \geq a_{i+1}$, even if if $\mathcal{P}$ lost an element in the first step, the same applies to the probability of connecting some $P_2 \in \mathcal{P'}$ to $\tilde{X}$ in the second step (as $|\mathcal{P'} | \geq a_i$).

Hence, with probability at least $1-2/\mathrm{e} > 0$ we can
connect $P_1\neq P_2 \in \mathcal{P}$ to $\tilde{X}$ via some
edges $e_1,e_2$. If this event occurs, delete $P_1,P_2$ from
$\mathcal{P}$, and replace them by $P = P_1 P_{\tilde{X}} P_2$,
where $P_{\tilde{X}}$ is a shortest path in $\tilde{X}$ between the two corresponding endpoints of the edges $e_1,e_2$. The vertices of $P_{\tilde{X}}$ are added to $I$, which is the set of vertices that are deleted from $X$.
Recall that $|P_{\tilde{X}}|$ is at most the diameter of $\tilde{X}$, which is at most $3\log m \leq 3\log n$. Combining this with the assumption that $|\mathcal{P}| \leq n/(h^4\log n)$, we deduce that during the whole process, the cardinality of $I$ is always bounded by $3n/h^4 = o(m)$. Thus, Lemma \ref{lem-resilient-expander} ensures us that the diameter of $\tilde{X}$ remains at most $3\log m$.

Since the above process costs $2T_i$ rounds (each of the two steps utilizes $T_i$ rounds), and succeeds in connecting two paths of $\mathcal{P}$ with probability at least $1-2/\mathrm{e} > 0$, standard Chernoff type bounds imply the following. For some absolute constant $c> 0$, performing this process $h \cdot a_i$ times (for a total of $2T_i \cdot h \cdot a_i$ rounds) decreases the size of $\mathcal{P}$ by at least $a_i / 2$ (that is, we get $|\mathcal{P}|\leq a_{i+1}$) with probability at least $1-\exp(-c h a_i)$.

Altogether, performing the above process for all $0\leq i \leq \lfloor\log_2 L\rfloor-1$, for a total of
$$ \sum_i 2 T_i \cdot h \cdot a_i \leq 4\lfloor\log_2 L\rfloor \frac{nh^3}{K} \leq 4 n/h^7 $$
rounds, concatenates $\mathcal{P}$ into a single path $P$ with failure probability at most
$$\sum_i \exp(-c h a_i) \leq 2\exp(-c h) = o(1)~.$$
At this point, we arbitrarily identify the endpoints of $P$ as \emph{left} and \emph{right}, and wish to join these two endpoints via a simple path in $\tilde{X}$. To this end, we essentially repeat the above process. A calculation similar to the one above shows that performing $T = \lfloor n/(2h^7)\rfloor$ rounds in an attempt to connect a given endpoint of $P$ to $\tilde{X}$ has a success probability of
$$\Big(1-\frac{2|\tilde{X}|}{n^2}\Big)^{K T} \geq
1-n^{-(1-o(1))h}=1-o(1)~.$$ Therefore, $2T \leq n/h^7$ further rounds enable us to close $P$ into a cycle \whp. Altogether, we have used $5n/h^7$ rounds in order to produce \whps a simple cycle going through every vertex of $Y$, as required.
\end{proof}

\subsection[Phase 4: creating a Hamilton cycle on the remaining vertices]{Phase \ref{item-small-ham-cycle}: creating a Hamilton cycle on the remaining vertices}\label{sec34}
At this point, there is a simple cycle going through every vertex of $Y$; once again, let $X = Y^c$ denote the remaining set of vertices, and recall that $|X| = (1-o(1))m$ (initially, $Y$ had cardinality $n-m$, to which we added a total of $o(m)$ vertices in connector paths).

In our model, at every given round, $K$ ordered pairs are uniformly and independently chosen out of $[n]^2$,
hence each of these pairs has a probability of $(1-o(1))(m/n)^2$ to describe an edge with both endpoints in $X$. By standard concentration arguments, at least $\frac{1}{2} K(m/n)^2 = \frac{1}{2}h^6\log n$ such edges appear in a given round with probability at least $1 - n^{-\Omega(h^6)}$. In particular, along any given $n$ rounds, $\whp$ every round contains at least $\frac{1}{2}h^6 \log n$ edges with both endpoints in $X$.

Therefore, restrict the process to the set $X$, with, say, $K' = \log n / h$ edges per round. In  this setting, Theorem \ref{thm-sublog} provides an algorithm for constructing a Hamilton cycle on $X$ within $(1+o(1))\frac{m}{2K'} \log m = (\frac{1}{2}+o(1))(m h) < m h$ rounds (for a sufficiently large $n$) \whp, as required.

\subsection[Phase 5: merging the two cycles]{Phase \ref{item-join-cycles}: merging the two cycles}
Consider the two simple cycles $C_X,C_Y$ on the vertices of $X,Y$ respectively, constructed in the previous phases, and choose an arbitrary orientation for each of them. For any vertex $x \in X$ and any vertex $y\in Y$, let $x^+$ and $y^+$ denote the subsequent vertices on the cycles $C_X$ and $C_Y$ respectively.
In order to patch the two cycles into one, we do the following.

First, we perform $T_1\deq n/h^2$ rounds, preferring edges in $e(X,Y)$, the cut between $X$ and $Y$, whenever such appear (settling ambiguities arbitrarily). Let $\mathcal{E}\subset e(X,Y)$ denote all edges between $X,Y$ added in this manner.
At every given round,
\begin{align*}\P(\mbox{receiving an edge of $e(X,Y)$ }) &= 1 - \left(1-\frac{2|X||Y|}{n^2}\right)^K
= 1 - \left(1-\frac{(2-o(1))m(n-m)}{n^2} \right)^K \\&\geq 1-\mathrm{e}^{-(2-o(1))mK/n}=1-n^{-\Omega(h^8)}~,\end{align*} hence at the end of $T_1$ rounds as above, by standard Chernoff-type inequalities, the probability of the event that $|\mathcal{E}| \geq n/(2h^2)$ is at least $1-\exp\left(-\Omega(n/h^2)\right)$. Condition therefore on this event.

Second, let
  $$\mathcal{E}^+ \deq \{ (x^+,y^+) : e=(x,y) \in \mathcal{E}\}~,$$
  and note that $|\mathcal{E}^+| = |\mathcal{E}|\geq n/(2h^2)$. We perform $T_2\deq n/\log n$ rounds, preferring edges in $\mathcal{E}^+$ and settling ambiguities arbitrarily. Letting $B$ denote the event of missing every edge of $\mathcal{E}^+$ throughout these $T_2$ rounds, we get
$$ \P(B) \leq 1-\left(1-\frac{2|\mathcal{E}^+|}{n^2}\right)^{K T_2} \leq \exp\left(-2|\mathcal{E}^+|K T_2 /n^2\right) \leq \exp\left(-h^8\right)=o(1)~.$$
Therefore, \whp, for some $x\in X$ and $y\in Y$ the above process adds the two edges $e=(x,y)$ and $e^+=(x^+,y^+)$, using which the two cycles $C_X$ and $C_Y$ can be patched in the obvious manner into a Hamilton cycle. This completes the proof of Theorem \ref{thm-suplog}.

\section[Intermediate regime]{Intermediate regime: $K=\Theta(\log n)$}\label{sec:intermediate}

In this section, we prove Theorem \ref{thm-log}. The lower and upper bounds we present are straightforward corollaries of our results in the previous section, and differ by a multiplicative factor of $3$. While we make no effort to tighten this gap, it seems that the techniques we used in the previous sections do not suffice for establishing a tight result for $K=\Theta(\log n)$.

We begin with the asymptotical upper bound, and recall that for this purpose, we may allow repeating edges and self-loops, and thus assume that the $K$ edges presented at each round are independently and uniformly selected from $[n]^2$.
The key element of the upper bound is the application of the
results of Section \ref{sec:bipartite}, where a bipartite random
graph of large (one sided) minimal degree was constructed. The
previous analysis of the second stage of the \greedy algorithm
deployed in that section will suffice for our purposes, yet when
analyzing its first stage we made use of the fact that we have
$\omega(n)$ rounds at our disposal (since $K$ was $o(\log n)$),
and this is no longer the case. We therefore require a more
refined result, which is incorporated in the next lemma. There, we
use the notion of a random $d$-out graph, this is a graph on $n$
labeled vertices, obtained by choosing, independently for every
vertex, a set of $d$ out-neighbors, and then by erasing edge
directions and by deleting multiple edges in case they appear.

\begin{lemma}\label{cor-rand-d-out}
Let $d$ be a fixed integer, let $K=K(n)$ grow to infinity with
$n$, and consider the Achlioptas process that presents $K$ edges at each round. Then a random $d$-out graph can be constructed in at most $(1+o(1))(d + \frac{\log n}{K})n$ rounds \whp, where the $o(1)$-term tends to $0$ as $n\to\infty$.
\end{lemma}
\begin{proof}
Recalling that each round consists of $K$ ordered pairs, each
chosen uniformly at random from $[n]^2$, we will ignore the second
coordinate of each pair, and base our decisions solely on the
degrees of the first coordinates (rounds featuring either a
repeated edge or a self-loop are automatically ignored). Clearly,
such an algorithm, guaranteeing minimum degree $d$ for every
vertex,
immediately provides a construction of a random $d$-out graph (one can simply take the \emph{first} $d$ out-neighbors assigned to each vertex).

Consider the first stage of the \greedy algorithm, which for $j\in\{0,\ldots,d-1\}$ prioritizes vertices of degree $j$ for a period of $T_1 = \left(1+\frac{\epsilon}{2d}\right)n$ rounds. That is, for every $j\in\{0,\ldots,d-1\}$, the algorithm performs $T_1$ rounds, where it chooses an ordered pair whose first coordinate is a vertex of degree $j$ whenever such a pair appears (settling ambiguities randomly). Whenever no such pair appears, the round is forfeited.

As before, let $X_j=X_j(t)$ denote the set of vertices that have degree $j$, and whenever the context of $X_j$ is clear, let $A_{X_j}$ denote the event that a pair, whose first coordinate belongs to $X_j$, is presented at a given round. Consider the phase where the algorithm focuses on $j$-degree vertices. As long as $|X_j|\geq n/\sqrt{K}$, the following holds:
$$\P(A_{X_j}) = 1- (1-|X_j|/n)^K \geq 1- \exp\left(-K|X_j|/n\right) \geq 1-\exp(-\sqrt{K}) \geq 1-\frac{\epsilon}{4d}~,$$
where the last inequality holds for any sufficiently large $n$, given the fact that $K\to\infty$ with $n$.
Hence, either $|X_j(T_1)| \leq n/\sqrt{K}$, or the number of
rounds in which we witness an ordered pair whose first coordinate
is a $j$-degree vertex stochastically dominates a binomial variable with $T_1$ trials and mean $(1+\frac{\epsilon}{2d})(1-\frac{\epsilon}{4d})n \geq (1+\frac{\epsilon}{8d})n$ (assuming that $\epsilon < 1/64$).
In this case, Chernoff's inequality implies that, along these $T_1$ rounds, we witness at least $n$ such ordered pairs $\whp$, and in particular, we again obtain that $|X_j(T_1)| \leq n/\sqrt{K}$ \whp. Altogether, we may condition on the event that $|X_j(T_1)| \leq n/\sqrt{K}$.

At this point, we can rejoin the analysis of Lemmas \ref{lem-greedy-stage-1} and \ref{lem-greedy-stage-2}. Recall that Lemma \ref{lem-greedy-stage-1} aimed to reduce the fraction of low-degree vertices to $\frac{\epsilon}{2K}$ at the cost of $\Theta(n)$ rounds, while Lemma \ref{lem-greedy-stage-2} eliminated all low-degree vertices at the cost of $(\frac{1}{2}+\epsilon)\frac{n}{K}\log n$ additional rounds. In our case, we could not afford the cost of $\Theta(n)$ rounds utilized by Lemma \ref{lem-greedy-stage-1}, hence we worked to reduce the initial fraction of low-degree vertices to $1/\sqrt{K}$. As we next show, this would allow the argument of Lemma \ref{lem-greedy-stage-1} to ensure the same end-result at the permissible cost of $\Theta(n/\sqrt{K})$ rounds. Furthermore, the analysis of Lemma \ref{lem-greedy-stage-2} would thereafter hold as is (up to a factor of $2$, owed to the fact that we only examine the first coordinate of each ordered pair).

Notice that henceforth, as long as $|X_j| > \frac{\epsilon}{2dK} n$, we have
$$\P(A_{X_j})\geq 1-\exp\left(-K|X_j|/n\right) \geq 1-\exp\left(-\frac{\epsilon}{2d}\right)~,$$
precisely bound \eqref{eq-At-equation-first-st}. Therefore, the argument following that equation in Lemma \ref{lem-greedy-stage-1} implies that, for some fixed $c_1>0$, after at most $\Delta = c_1 |X_j(T_1)|$ rounds we get $|X_j|<\frac{\epsilon}{2dK}n$ \whp. Recalling that $X_j(T_1) \leq n/\sqrt{K}$ and that $K\to\infty$ with $n$, we deduce that $\Delta \leq \frac{\epsilon}{2d}n$ for every sufficiently large $n$.

Altogether, after $T_1+\Delta \leq \left(1+\frac{\epsilon}{d}\right)n$ rounds we obtain that $|X_j|<\frac{\epsilon}{2dK}n$ \whp. By applying this sequentially for $j=0,\ldots,d-1$, almost surely all but at most $\frac{\epsilon}{2K}n$ vertices obtain degree at least $d$ (with uniformly distributed neighbors) within a total of $(d+\epsilon)n$ rounds.

The second stage of the algorithm, corresponding to Lemma \ref{lem-greedy-stage-2}, focuses only on the vertices $\bigcup_{j=0}^{d-1} X_j$, that is, the vertices which have degree smaller than $d$ at the end of the first stage. Call this set of vertices $U$. The \greedy algorithm at this stage will choose an ordered pair, whose first coordinate is in $U$, whenever one is presented (settling ambiguities randomly, and ignoring rounds where no such pair appears).

Let $A_U$ denote the event that, at a given round, we witness a pair whose first coordinate lies in $U$
(recall that \eqref{eq-At-equation-second-st} gave a bound on the analogous probability in Lemma \ref{lem-greedy-stage-2}). As $|U| < \frac{\epsilon}{2K}n$, we have
$$\P(A_U) = 1-\left(1-\frac{|U|}{n}\right)^K \geq
\left(1-\frac{K|U|}{2n}\right)K\frac{|U|}{n} \geq
\left(1-\frac{\epsilon}{4}\right)K\frac{|U|}{n}~.$$
Therefore, applying the same argument as in Lemma \ref{lem-greedy-stage-2}, only with $T_2 = (1+\epsilon)\frac{n}{K}\log n$, provides a minimal degree of $d$ after $T_2$ additional rounds \whp.
\end{proof}

The upper bound of $(1+o(1))(3+\frac{\log n}{K})n$ now immediately follows from a beautiful result of Bohman and Frieze \cite{BohF}, which states that a random $3$-out graph is Hamiltonian \whp.

The lower bound will follow from the next lemma:
\begin{lemma}
  Let $d$ be a fixed integer, and let $\epsilon > 0$. Consider the Achlioptas process which presents $K=K(n)$ edges at each round. Then for any edge-choosing online algorithm, after $T=(1-\epsilon)(d +\frac{\log n}{K})n/2$ rounds there remain $n^{\epsilon/2}$ vertices of degree smaller than $d$ \whp.
\end{lemma}
\begin{proof}
Consider the graph after $T=(1-\epsilon)dn/2$ rounds of the process have been completed; clearly, at this point (by a simple counting argument) there are at least $\epsilon n$ vertices of degree smaller than $d$ in the graph.
Let $X$ denote this set of vertices. We claim that, \whp, at least
$n^{\epsilon / 2}$ of these vertices will not be incident to any
of the edges that appear in the next $\Delta=(1-\epsilon)\frac{n}{2K}\log n$ rounds altogether. This follows from a straightforward second moment argument, identical to the one that shows that  the random graph $\mathcal{G}(n,(1-\epsilon)\frac{\log n}{n})$ has $n^{\epsilon-o(1)}$ isolated vertices (and in fact, this is precisely the graph obtained by collecting all $K \Delta$ edges featured along the $\Delta$ rounds).

Once again, to simplify the analysis, we assume that the input of each round is a sequence of $K$ ordered pairs, each chosen uniformly at random and independently from $[n]^2$. As argued before, the number of rounds where this selection contains an ``illegal'' pair (a repeating edge or a self loop) is negligible.

For $v \in X$, let $Y_v$ denote the indicator of the event that none of the $n-1$ potential edges incident to $v$ appear in any of the $\Delta$ additional rounds we make. The following holds:
$$ \E Y_v = \Big(1-\frac{2(n-1)}{n^2}\Big)^{K \Delta }\geq \mathrm{e}^{-(2-o(1)) K \Delta/n} =  n^{-1+\epsilon+o(1)}~.$$
Let $Y = \sum_{v\in X}Y_v$. We obtain that, for instance, $\E Y \geq 2n^{\epsilon/2}$ for any sufficiently large $n$. Similarly, for any $u,v\in X$
\begin{align*} \Cov(Y_u,Y_v) &= \E\left[Y_u Y_v\right]-\E Y_u \E Y_v =\left(1-\frac{2(2n-3)}{n^2}\right)^{K \Delta} -
\left(1-\frac{2(n-1)}{n^2}\right)^{2K \Delta} < 0~.\end{align*}
The last inequality holds for any $n\geq 4$, since for any such $n$ we have
$\left(1-\frac{2(2n-3)}{n^2}\right) < \left(1-\frac{2(n-1)}{n^2}\right)^{2}$.
Therefore, $\var(Y) \leq \E Y$, and by Chebyshev's inequality we have $Y > n^{\epsilon/2}$ \whp, as required.
\end{proof}
The above lemma implies that after $(1-o(1))(1+\frac{\log
n}{2K})n$ rounds of the Achlioptas process, the obtained graph
\whps contains vertices of degree smaller than 2, and is thus
not Hamiltonian. This concludes
the proof of Theorem \ref{thm-log}.

\section{Concluding remarks and open problems}\label{sec:conclusion}
In this section we describe briefly several related results that
can be obtained using the methods of the current paper, and also
discuss some related problems.

\begin{asparaenum}[\bf 1.]
\item It is quite natural to ask about the validity of
the hitting time version of our result. Specifically, the question
is: given the value of $K\ge 2$, does there exist an online
algorithm for the Achlioptas process with parameter $K$ that is
capable of creating \whp\ a Hamilton cycle {\em exactly} at the
moment (round) where for the first time every vertex is incident
to at least two edges in the union of edges presented at all
rounds? In quite a few random graph processes the hitting time of
the property of being of minimum degree at least two coincides
\whp\ with that of Hamiltonicity. In our case, we have reasons to believe
that the picture is different even for $K=2$. Here is an heuristic
argument supporting this belief of ours. Consider indeed a typical
moment when the last vertex of degree at most one disappears in
the union of all presented edges. This moment comes normally after we
have seen about $m=\frac{n}{2}(\log n+\log\log n)$ edges, that is,
after  about $m/2$ rounds. At this moment the number of vertices
of degree exactly two will be of order $\log n$. In order for the
hitting time result to be valid, all edges incident to these
vertices of degree two should have been chosen by the algorithm in
corresponding rounds. There are $\Theta(\log n)$ of these edges,
and therefore with decent probability one of them appeared during
the first $n/2$ (say) rounds. Denote this edge by $e$, and its
counterpart in the corresponding round by $f$. Again, with decent
probability both $e$ and $f$ were isolated edges at that round,
and therefore the algorithm could not really distinguish between
them and had no reasons to choose $e$ over $f$ at that round.
Thus, the hitting time version of our result appears to be rather
problematic.

\item
A setting closely related to the Achioptas
process is that of {\em online Ramsey} problems. In this setting,
similarly to the Achlioptas process with parameter
$K$, each round an online algorithm is presented with $K$ edges,
chosen uniformly at random from the set of edges of the complete
graph $K_n$ on $n$ vertices. (The difference with our setting,
where the edges are chosen only from those missing in the current
graph, is usually insignificant.) The algorithm -- unlike in the
Achlioptas process, where only one edge is to be chosen and the
rest are discarded -- colors the $K$ presented edges in $K$
distinct colors. Usually a graph property $P$ (Hamiltonicity, existence of
a copy of a fixed graph $H$, etc.) is given, and the algorithm's
goal is either to create a graph possessing $P$ in {\em each} of
the $K$ colors as soon as possible, or alternately to avoid
creating $P$ in any of the colors for as long as possible.
(One should mention that the above described setting is just one
of several possible Ramsey-type online games, another possible
setting is where edges arrive one by one and are colored in one of
$K$ colors, this setting has been considered by Marciniszyn,
Sp\"ohel and Steger in \cite{MSS}.) In our context, the property
$P$ under consideration is that of Hamiltonicity, and the
algorithm's task is to create a Hamilton cycle in each of the $K$
colors. This is certainly a harder task than creating a Hamilton
cycle in the Achlioptas process -- the latter corresponds
essentially to creating a Hamilton cycle in the first color.

Using our techniques, we can solve the above described
problem asymptotically for the case where $K=o(\sqrt{\log n})$. For this case, we can
describe an algorithm that \whp\ creates a Hamilton cycle in each
of the $K$ colors during $\frac{1+o(1)}{2K}n\log n$ rounds, thus
strengthening our main result for this range of the parameter $K$.
Here is a sketch of the proof. At large it is quite close to the
proof presented in Section \ref{sec:sublog}, so we restrict
ourselves to describing the required adjustments in our argument.
In the text below, $c,c'$ stand for generic positive constants
whose values can be adjusted appropriately from an occasion to an
occasion for the argument to go through.

Just like in Section \ref{sec:sublog} the algorithm proceeds in
three stages. In the first stage, it aims to create a vertex
subset $W$ of size $|W|\geq (1-\epsilon)n$ such that the subgraph
spanned by $W$ in {\em each} of the $K$ colors
is an expander. In order to
achieve this goal, each round the algorithm colors $K$ presented
edges at random into $K$ distinct colors. The first stage lasts
$\frac{3K}{4\epsilon}n$ rounds. Putting $D=K/\epsilon$,
each color class is then distributed as a
random graph $\mathcal{G}(n,p)$ with $p=\frac{3D}{2n}$.
In the proof of Lemma \ref{lem-d-core} we have shown that with probability $1-\mathrm{e}^{-\Omega(n)}$, for every subset $S$ of $\mathcal{G}(n,\frac{3D}{2n})$ of size $n/D$, there are at least $n$ edges in the cut $(S,S^c)$. A union bound thus implies that \whp\ this holds for all the $K$ color classes at once. Next, conditioning on this event, we perform the following iterative process. Starting with $W$ as the entire set of vertices, we repeatedly remove from $W$ (in an arbitrary order) any vertex that has less than $D$ neighbors in $W$ in one of the $K$ color classes. Notice that, once $n / D$ vertices are removed on account of some given color class, they form a set $S$ with $|\partial S|<n$ in that color class, contradicting our assumption. It thus follows that the process ends after at most $K n/D = \epsilon n$ vertices are removed from $W$. Therefore, the resulting subset $W$ has size $|W|\geq(1-\epsilon)n$ and each of its vertices has at least $D$ neighbors in $W$ in every color.

At this point, we apply Corollary \ref{cor-set-expansion-in-H}, which to be precise applies not only to the $D$-core of a random graph $\mathcal{G}(n,p)$ with
$p=\frac{3D}{2n}$, but rather to any subgraph of this random graph that has minimal degree $D$,
and its statement holds with probability $1-n^{-\Omega(1)}$ (as the calculation in Lemma \ref{lem-small-avg-deg} shows). We deduce that for any given color, with probability $1-n^{-\Omega(1)}$, every subset $S\subset W$ of size $1\leq s \leq n/100$ has at least $3s$ neighbors of the same color in $W \setminus S$. Thus, the above statement holds \whp\ for all $K$ colors at once, and altogether,
the required expansion properties of $W$ are obtained in each
of the $K$ colors at the cost of $\Theta(K n)$ rounds. Using our assumption
that $K=o(\sqrt{\log n})$, this amount of rounds is $o\left((n/K)\log n\right)$,
as desired.

Set $U\deq V\setminus W$, and for $u\in U$ let $d(u,W)$ denote its number of neighbors
in $W$. The first part of Stage 2 of the algorithm is again similar to the second stage of the
current proof, but now the algorithm aims to have $d(u,W)\geq dK$ for
all but at most $\epsilon n/(2K)$ vertices of $U$, where
$d=20$. The algorithm effectively chooses at most one
edge per round and colors it in the required color, the colors
rotate at every vertex of $U$ (thus, the first chosen edge between
$u\in U$ and $W$ is colored in the first color, the second one in
the second color, etc.); this way we ensure that once $dK$ chosen edges
in the cut $(U,W)$ touch a vertex $u\in U$, this vertex has at
least $d$ neighbors uniformly chosen over $W$ in each of the $K$ colors. The argument
here is quite similar to that of Subsection \ref{sec:bipartite},
but now we aim at $|X_j|\le \frac{\epsilon n}{2dK^2}$ for all $1 \leq j < dK$
(where $X_j$ is the number vertices $u\in U$ with $d(u,W)=j$). In order to analyze the process of the gradual decrease of
$|X_j|$, we say that a substage $i$ is completed when $|X_j|\le
\epsilon n/2^i$, $i=1,\ldots,\log n$. Call a round successful if one of
the $K$ presented edges is between $X_j$ and $W$. Recalling \eqref{eq-At-equation-first-st},
the probability
that a round is successful is at least
$\min\left(c,\frac{c|X_j|K}{n}\right)$. By the balls-and-bins argument in
the proof of Lemma \ref{lem-greedy-stage-1}, it follows that $O(|X_j|)$
successful rounds would complete substage $i$. Therefore, with very
high probability we need to
wait $\max\left(c'|X_j|, \frac{c'n}{K}\right)$ rounds. The total
waiting time for all substages then is
$$
\sum_{i=1}^{\log(dK^2+1)}\max \left(\frac{c'\epsilon n}{2^i},
\frac{c'n}{K}\right)=O(n)
$$
with probability exponentially close to 1. Summing over all $j$,
we see that \whp\ we need $O(K n)$ rounds to complete
the first part of Stage 2
(recall that $d$ is a constant). Recalling our assumption
$K=o(\sqrt{\log n})$, we again derive that \whp\ the first part of
Stage 2 completes successfully in $o\left((n/K)\log n\right)$ rounds.

The second part of
Stage 2 of the algorithm starts with a residual set $U_0\subset
V\setminus W$ of cardinality $|U_0|=O(\epsilon n/K)$ such that
all vertices of $U\setminus U_0$ have at least $d$ random neighbors in $W$ in each of the $K$ colors. The algorithm will choose and
color effectively at most one edge per round.
Here the goal is to ensure that by the end of the stage each vertex
$u\in U_0$ will have degree at least $dK$ into $W$ (and then, just like
before, we will color these edges while rotating colors). The
argument is very similar to that of Lemma \ref{lem-greedy-stage-2}.
Denote $|U_0|=t$. During
$T_2=(\frac{1}{2}+\epsilon)\frac{n}{K}\log n$ rounds we observe and
color $M=(1+\frac{\epsilon}{3})t\log n$ edges between $U_0$ and
$W$ with probability $1-n^{-\Omega(t)}$, just as in Lemma
\ref{lem-greedy-stage-2}. The final calculation is a bit different
here: the probability that after $T_2$ rounds a vertex $v\in U_0$
has less than $dK$ edges incident to $v$ and chosen by the
algorithm is at most:
\begin{align*}
d K \binom{M}{d K} \left(\frac{1}{t}\right)^{dK}
\left(1-\frac{1}{t}\right)^{M-dK}
\leq dK \left(\frac{eM}{dKt}\right)^{dK}\mathrm{e}^{-(1-o(1))M/t}\leq dK (c\log n)^{dK}\mathrm{e}^{-\left(1+\frac{\epsilon}{3}\right)\log n}\ .
\end{align*}
Recalling that we assume $K=o(\sqrt{\log n})$ (a much weaker
assumption would suffice here), the latter estimate is $o(1/n)$,
and we can apply the union bound to derive that \whp\ after
$\frac{1}{2}(1+\epsilon)\frac{n}{K}\log n$ rounds the second part
of Stage 2 will be
completed with all vertices of $U$ having at least $d$ neighbors
in $W$ in each of the $K$ colors.

Stage 3 is essentially identical to the corresponding stage of our
algorithm for the Achlioptas process, as described in Section
\ref{rot-ext}. Currently we have an expander in each of the $K$
colors, and as argued in Section \ref{rot-ext}, adding a linear
number of random edges on top of each color produces with very
high probability a Hamilton cycle. Running the process for $Cn$
additional rounds, with $C>0$ being a large enough constant, and
coloring the $K$ edges of each round in $K$ distinct colors at
random meets the above goal.

\item Our techniques can be easily adapted to prove
the following result about the Achlioptas processes: for $t=O(1)$
and $K=o(\log n)$, there exists an online algorithm for the
Achlioptas process with parameter $K$ that creates \whp\ a spanning
$t$-connected graph in $\frac{1+o(1)}{2K}n\log n$ rounds. Moreover, a
similar adaptation of the argument presented above allows to
derive a Ramsey-type result for the property of creating a
$t$-connected spanning subgraph in that many rounds as long as
$K=o(\sqrt{\log n})$.

\item A closely related property to be considered for
Achlioptas/Ramsey processes is that of the existence of a {\em
perfect matching} (assume the number of vertices $n$ is even). For
$K=o(\log n)$, our main result for this regime yields also an
algorithm producing \whp\ a perfect matching in $\frac{1+o(1)}{2K}n\log n$
rounds, and this is clearly asymptotically optimal (for the same reasons
our Hamiltonicity result is asymptotically optimal). For the
regime $K=\omega(\log n)$, it is possible to create a perfect
matching \whp\ in $(1+o(1))n/2$ rounds as follows. First, we
greedily construct an almost perfect matching $M_1$; then, on the
vertices uncovered by $M_1$ we construct a relatively dense random
graph (similarly to our argument from Section \ref{sec34}) that
will contain \whp\ a perfect matching $M_2$, due to standard
results from the theory of random graphs. The union of $M_1$ and
$M_2$ will then form a perfect matching.
\end{asparaenum}

\end{document}